\documentclass[final,3p,times]{elsarticle}

\usepackage{amssymb}
\usepackage{amsmath,amssymb,amsfonts,amsthm}
\usepackage{latexsym}
\usepackage{graphicx}
\usepackage{bm}
\usepackage[]{subfig}
\usepackage{fancyhdr}
\usepackage{booktabs}
\usepackage{caption}
\usepackage{physics}
\usepackage{color}
\definecolor{hillencolor}{rgb}{0.66, 0.030, 0.30}

\newtheorem{theorem}{Theorem}%
\newtheorem{lemma}[theorem]{Lemma}%
\newtheorem{definition}{Definition}%
\newproof{Proof}{proof}

\begin{document}
	
	\begin{frontmatter}
		
		
		
		\title{Positivity and global existence for nonlocal advection-diffusion models of interacting populations}
		
		
		\author{Valeria Giunta\corref{cor1}\fnref{label1}}
		\cortext[cor1]{Corresponding author}
		\author[label2]{Thomas Hillen}
		\author[label3]{Mark A. Lewis}
		\author[label4]{Jonathan R. Potts}
		\affiliation[label1]{School of Mathematics and Computer Science, University of Swansea, Computational Foundry, Crymlyn Burrows, Skewen, Swansea SA1 8DD, UK}
		\affiliation[label2]{Department of Mathematical and Statistical Sciences, University of Alberta, Edmonton,
			AB T6G 2G1, Canada}
		\affiliation[label3]{Department of Mathematics and Statistics and Department of Biology, University of Victoria,
			PO Box 1700 Station CSC, Victoria, BC, Canada} 
		\affiliation[label4]{School of Mathematics and Statistics, University of Sheffield, Hicks Building, Hounsfield Road,
			Sheffield S3 7RH, UK}
		
		\begin{abstract}
			In this paper we study a broad class of non-local advection-diffusion models describing the behaviour of an arbitrary number of interacting species, each moving in response to the non-local presence of others. Our model allows for different non-local interaction kernels for each species and arbitrarily many spatial dimensions. 
			We prove the global existence of both non-negative weak solutions in any spatial dimension and positive classical solutions in one spatial dimension. These results generalise and unify various existing results regarding existence of non-local advection-diffusion equations. We also prove that solutions can blow up in finite time when the detection radius becomes zero, i.e. when the system is local, thus showing that nonlocality is essential for the global existence of solutions. We verify our results with some numerical simulations on 2D spatial domains. 
			
		\end{abstract}
		
		
		
		\begin{keyword} Nonlocal advection; global existence; blow-up.
			
			
			
		\end{keyword}
		
	\end{frontmatter}
	
	
\section*{Introduction}

\noindent We consider a multispecies model of interacting species, which sense their environment and other species in a nonlocal way \cite{fetecau2011,PL19,PPH}. The individual populations are denoted by $u_i(x,t)$, where $t\geq 0$ denotes time, $x\in \Omega$ denotes space and the index $i=1,\dots,N$ denotes the species.  The model is given by 
\begin{equation}\label{eq:model}
	u_{it} = D_i \Delta u_i +\nabla\cdot\left( u_i \sum_{j=1}^N \gamma_{ij} \nabla (K_{ij} \ast u_j) \right), \qquad i=1,\dots,N. \end{equation}
Here, $K_{ij} $ is a twice-differentiable function,
with $\nabla K_{ij} \in L^\infty$, and $K_{ij} \ast u_j$ denotes a convolution operator defined as 
\[ K_{ij}\ast u_j (x) = \int_\Omega K_{ij}(x-y) u_j(y) dy.\]
From a biological perspective, $K_{ij}$ describes the nonlocal sensing of species $j$ by species $i$. The constants $D_i>0$ are diffusion coefficients of species $i$ and the values of $\gamma_{ij}$ denote the extent to which species $i$ avoids (if $\gamma_{ij}>0$) or is attracted to (if $\gamma_{ij}<0$) species $j$. For the definition of the nonlocal term to make sense, the domain $\Omega$ cannot have boundaries.  Here we let $\Omega=\mathbb{T}^n$, the $n$-torus defined by identifying the boundaries of $[-L_1,L_1]\times\dots\times[-L_n,L_n]$ in a periodic fashion.

Nonlocal interaction models, such as (\ref{eq:model}), have become an important modelling tool in the mathematical modelling of biological species \cite{lewis2016mathematics,eftimie2018hyperbolic,burger2018sorting,buttenschon2021non,wangsalmaniw2022,PPH}. Organisms do not typically make movement decisions only based on the local information they have about prey, predator, or food sources. Rather, movement decisions are based on information gathered over a certain `perceptual radius' via sight, smell, sounds, or other means of sensing \cite{pottslewis2019}.  Currently, the mathematical analysis of organism movement based on nonlocal perception is at full swing \cite{painter2024biological}. Several authors consider models of the form (\ref{eq:model}) to study species aggregation, segregation, avoidance, home ranges, territories, mixing, and spatio-temporal patterns \cite{GVA06,eftimieetal2007,Carrillo-chapter,burger2018sorting,PL19,glimm2020b,Carrillo2020,carrillo2021phase,potts2022beyond,giunta2022detecting,jewell2023,PPH}. In many of these papers, the analysis starts with results on local and global existence and positivity. These results are generated through various methods, such as energy functionals \cite{giunta2022detecting,Carrillo2020,Jungel2022,PPH}, semigroup theory and fixed-point arguments \cite{buttenschon2021non}, or direct PDE-type estimates \cite{hillen2018}, all depending on the specific model at hand.  In \cite{Carrillo2020}, using methods from \cite{Chazelle2017}, the authors consider model (\ref{eq:model}) for one species with smooth interaction kernel and they show global existence of classical solutions, using energy-entropy methods. In \cite{Jungel2022} the assumption of smooth interaction kernels is relaxed, and assuming a detailed balance condition on the kernels, global existence for weak solutions is shown. Our previous work in \cite{giuntaetal2022a} proves existence of local solutions in any space dimension and global solutions in 1D for the case of equal interaction kernels, i.e. $K_{ij}=K$.

Here we combine these results into a unifying existence theory for nonlocal models of type (\ref{eq:model}). In contrast to previous models, we allow the interaction kernels $K_{ij}$ to vary from species to species, and we have no restriction on the space dimension. This is somewhat surprising, since global existence for local versions of our model do strongly depend on the spatial dimension  \cite{Gilbarg,Lieberman,Suzuki}. This behavior is similar to the well known chemotaxis model \cite{horstmann2004,UsersGuide}. Solutions to the standard chemotaxis model are globally bounded in 1-D, while they blow-up in higher dimensions, if the initial population is large enough in the $L^{n/2}$-norm, where $n$ denotes the space dimension \cite{horstmann2004}. Guided by these observations, we consider the local limit of (\ref{eq:model}) and we also find cases in $n \geq 2$ dimensions where finite time blow-up is possible.

The paper is organised as follows. In Section \ref{sec:1} we define a modified version of Equation (\ref{eq:model}) and prove some preliminary results. This modified model is then analysed in Section \ref{sec:mild_sol}, where we prove the local existence of mild solutions, and in Section \ref{sec:global_existence}, where we prove the global existence of positive solutions. We conclude our proof by showing that every positive solution of the modified model is also a solution of Equation (\ref{eq:model}). In Section \ref{sec:blowup} we show that in the corresponding local system the solutions can blow up in finite time. Section \ref{eq:num} concludes with numerical simulations showing that the solutions of the nonlocal problem, although they become steeper as the detection radius becomes smaller, still remain bounded.

\section{A modified version of our system}\label{sec:1}
\noindent To establish our existence results for Equation (\ref{eq:model}), our approach is first to prove existence and non-negativity of weak solutions to a slightly modified version of Equation (\ref{eq:model}). We then show that any solution of this modified system is also a solution of Equation (\ref{eq:model}).  The modified system is as follows
\begin{align}\label{eq:model2}
	u_{it} = D_i \Delta u_i +\nabla \cdot \left( h(u_i) \sum_{j=1}^N \gamma_{ij} \nabla(K_{ij} \ast u_j) \right), \qquad i=1,\dots,N,
\end{align}
where $h(u) = u$ if $u\geq 0$ and $h(u)=0$ if $u<0$. 
Note that whenever $u_i\geq 0$, Equations (\ref{eq:model}) and (\ref{eq:model2}) are identical.  In Equation (\ref{eq:model2}), derivatives are understood weakly.  In particular, the weak derivative of $h(u)$ is $h'(u) = 1$ if $u>0$ and $h'(u)=0$ if $u<0$. 
We collect some basic properties of  $h(u)$ in the following Lemma.
\begin{lemma} \label{lem:hu1} For any $v \in  H^1({\mathbb T})$, we have $\|h(v)\|_{L^2}\leq \|v\|_{L^2}$, $\|\nabla h(v)\|_{L^2}\leq \|\nabla v\|_{L^2}$, and $\|h(v)\nabla v\|_{L^1}\leq \|v\nabla v\|_{L^1}$.  For any $v_1,v_2 \in L^2({\mathbb T})$, we have $\|h(v_1)-h(v_2)\|_{L^2}\leq \|v_1-v_2\|_{L^2}$.\end{lemma}
\noindent{\bf Proof.}  The inequality $\|h(v)\|_{L^2}\leq \|v\|_{L^2}$ follows from the definitions of $h$ and the $L^2$-norm. The inequality $\|\nabla h(v)\|_{L^2}\leq \|\nabla v\|_{L^2}$ follows from the same definitions, and also that $\nabla h(v)=h'(v)\nabla v$. 
The inequality $\|h(v)\nabla v\|_{L^1}\leq \|v\nabla v\|_{L^1}$ follows from the definitions of $h$ and the $L^1$-norm. 

For the final inequality, we observe that 
\begin{align}
	\|h(v_1)-h(v_2)\|_{L^2}^2&=\int_{\mathbb T}(h(v_1(x))-h(v_2(x)))^2{\rm d}x\nonumber =\int_{S_1}(v_1(x)-v_2(x))^2{\rm d}x+\int_{S_2}v_1^2(x){\rm d}x+\int_{S_3}v_2^2(x){\rm d}x
\end{align}
where $S_1=\{ x \in {\mathbb T}:v_1(x)>0, v_2(x)>0\}$, $S_2=\{ x \in {\mathbb T}:v_1(x)> 0, v_2(x)\leq 0\}$, and $S_3=\{ x \in {\mathbb T}:v_1(x)\leq 0, v_2(x)> 0\}$.  Now, for $x \in S_2$, we have $v_2(x)\leq 0$ and $v_1(x)>0$ so $ v_1(x)\leq  v_1(x)-v_2(x)$, and then $v_1^2(x)\leq (v_1(x)-v_2(x))^2$. Similarly, for $x \in S_3$, we have $ v_2(x) \leq v_2(x)-v_1(x) $ and $v_2^2(x)\leq (v_2(x)-v_1(x))^2$.  Hence
\begin{align}
	&\int_{S_1}(v_1(x)-v_2(x))^2{\rm d}x+\int_{S_2}v_1^2(x){\rm d}x+\int_{S_3}v_2^2(x){\rm d}x \nonumber \leq \int_{S_1\cup S_2 \cup S_3}(v_1(x)-v_2(x))^2{\rm d}x \nonumber \leq \|v_1-v_2\|_{L^2}^2,
\end{align}
so that $\|h(v_1)-h(v_2)\|_{L^2}\leq \|v_1-v_2\|_{L^2}$. \qed




\section{Local existence of mild solutions}\label{sec:mild_sol}

\noindent We begin by proving local existence of mild solutions to Equation (\ref{eq:model2}).  
\begin{definition}
	Given $u_0=(u_{10},...,u_{N0})\in (L^2(\mathbb{T}^n ))^N$ and $T>0$, we say that 
	$u(x,t)=(u_{1}(x,t),...,u_{N}(x,t)) \in L^\infty((0,T), L^2(\mathbb{T}^n ))^N$
	is a  {\bf mild solution} of Equation \eqref{eq:model2} if 
	\begin{equation}\label{eq:mild}
		u_i = e^{D_i\Delta t} u_{i,0} -  \int_0^t e^{D_i\Delta (t-s)} \nabla \cdot \left[h(u_i) \nabla \left( \sum_{j=1}^N \gamma_{ij}K_{ij} \ast u_j\right)  \right] ds 
	\end{equation}
	for each $0<t\leq T$, where $e^{D_i\Delta t} $
	denotes the solution semigroup of the heat equation $u_{it} = D_i \Delta u_i$ on $\mathbb{T}^n $, and $u_{i,0}(x)=u_{i}(x,0)$ is the initial condition. 
\end{definition}  

\begin{theorem}\label{t:local} Assume $u_0\in H^2(\mathbb{T}^n)^N$ and each $K_{ij}$ is twice  differentiable with $\max_{i,j} \|{\nabla}K_{ij}\|_\infty<\infty$. For each $u_0 \in L^2(\mathbb{T}^n )^N$ there exists a time $T>0$ and a unique mild solution of Equation \eqref{eq:model2} with $u\in L^\infty((0,T), L^2(\mathbb{T}^n ))^N.$
	Moreover, $u\in C^1((0,T_*), L^2(\mathbb{T}^n ))^N \cap C^0([0,T_*), H^2(\mathbb{T}^n ))^N.$
\end{theorem}
{\bf Proof.} 
{\noindent} In \cite[Theorem 3.6]{giuntaetal2022a} we showed local existence of mild solutions for Equation (\ref{eq:model})  using a fixed point argument. In that case the sensing mechanism for each species was equal $K_{ij}=K$  for all $i,j=1,\dots,N$ where $K$ is twice differentiable. To prove the same result for variable $K_{ij}$ is straightforward. It   requires replacing each $\|\nabla K\|_\infty$ with $\max_{i,j=1,\dots,N} \|{\nabla}K_{ij}\|_\infty$. To prove this for Equation (\ref{eq:model2}) rather than Equation (\ref{eq:model})  requires additionally employing the estimates on $h$ from Lemma \ref{lem:hu1}.  Other than this, the proof remains unchanged from that in \cite[Theorem 3.6]{giuntaetal2022a} so we do not repeat it here.
\qed

\section{Global existence and positivity}\label{sec:global_existence}
\noindent Following the strategy of \cite{giuntaetal2022a}, we define a time $T_\ast$ as follows.  If $\|u\|_{L^1}$ is bounded for all time then let $T_\ast=\infty$.  Otherwise $\|u\|_{L^1}\rightarrow \infty$ as $t \rightarrow T_{\rm max}$ for some time $T_{\rm max} \in (0,\infty)$.  In this case, let $T_*$ be the earliest time such that $\|u\|_{L^1}=2\|u_0\|_{L^1}$. Our aim is to establish existence and positivity up to time $T_*$, then use this to prove that $T_*<\infty$ leads to a contradiction.  This means that $T_\ast=\infty$, establishing global existence of weak solutions.

{\noindent}To show that $T_*=\infty$, we need the following positivity result.
\begin{lemma}\label{l:positivity} Keep all assumptions from Theorem \ref{t:local} and also assume  $u_i(x,0)\geq 0$ in $\mathbb{T}^n$ for all $i=1,\dots,N$. Then $u(x,t)\geq 0$ for solutions of  Equation \eqref{eq:model2}  for $t \in (0,T_*)$. Here we understand $u\geq 0$ a.e. in $\mathbb{T}^n$ component-wise.
\end{lemma}
\noindent{\bf Proof.} Suppose $u=(u_1,...,u_N)$ is a solution to Equation (\ref{eq:model2}) and let us fix an index $i\in\{1,\dots,N\}$. We use a standard idea of cutting off the negative part of the solution. Such a method has, for example, been used  in \cite{hipa1} for chemotaxis models. We define the negative part as $u^-_i(x,t):= u_i(x,t)$ if $u_i(x,t)<0$ and $u^-_i(x,t):= 0$ if $u_i(x,t)\geq 0$, 
and we split the domain $\mathbb{T}^n$ as 
$J_-(t)= \{ x\in \mathbb{T}^n: u_i(x,t) <0\}, J_{\,0}(t) = \{ x\in \mathbb{T}^n: u_i(x,t) =0\},$ and	$J_+(t)= \{ x\in \mathbb{T}^n: u_i(x,t) >0\}.$ 
Since 
the $L^2$-norm of $u_i$ is differentiable in time, we can write 
\begin{align}
	\frac{d}{dt} \frac{1}{2}\|u_i^-(.,t)\|_2^2 &= \int_{J_-(t)} u_i^- u_{it}^- dx 
	+ \underbrace{\int_{J_0(t)} u_i^- u_{it}^- dx}_{=0} 
	+ \underbrace{\int_{J_+(t)} u_i^- u_{it}^- dx}_{=0} = \int_{J_-(t)} u_i^- u_{it}^- dx. \end{align}
Since $J_-(t)$ is an open set and $u_i$ and its weak spatial derivatives are continuous and differentiable in time, we have
$u_{it}^- = u_{it} $ and $\nabla u_i^- = \nabla u_i,$ on $J_-(t).$
Then from Equation (\ref{eq:model2}) we obtain 
\begin{align}
	\frac{d}{dt} \frac{1}{2}\|u_i^-(.,t)\|_{L^2}^2 &=
	\int_{J_-(t)} u_i^-\left(D_i \Delta u_i +\nabla \cdot\left(h(u_i) \sum_{j=1}^N \gamma_{ij} \nabla (K_{ij}\ast u_j) \right)\right)dx \nonumber \\
	&= -D_i\int_{J_-(t)} \lvert\nabla u_i^-\rvert^2 dx + \int_{\partial J_-(t)} u_i^- D_i (\nabla u_i \cdot \mathbf{n} ) dS \nonumber - \int_{J_-(t)} (\nabla u_i^-)\cdot\biggl[  h(u_i) \sum_{j=1}^N \gamma_{ij} \nabla (K_{ij}\ast u_j)\biggr] dx \nonumber \\
	& + \int_{\partial J_-(t)} u_i^- h(u_i) \sum_{j=1}^N \gamma_{ij} \bigl(\nabla (K_{ij}\ast u_j)\cdot \mathbf{n}\bigr) dS  
\end{align}
where $dS$ is used to denote the boundary measure on $\partial J_-(t)$ and $\mathbf{n}$ denotes the outward normal vector on $\partial J_-(t)$. On $\partial J_-(t)$, we have $u_i^-=0$, hence both boundary integral terms vanish. The third term on the right hand side also vanishes, since on $J_-(t)$ we have $h(u_i)=0$. Hence we find 
\[
\frac{d}{dt} \frac{1}{2}\|u_i^-(.,t)\|_{L^2}^2
= -2 D_i \left(  \frac{1}{2}\|{\nabla} u_i^-(.,t)\|_{L^2}^2\right){\leq 0 }.   \]
Therefore $\|u_i^-\|_2^2$ is a Lyapunov function and when  $ \| u_i^-(.,0)\|_{L^2}=0$ then $\|u_i^-(.,t)\|_{L^2}^2 = 0$ for all $t >0$.\qed
\begin{theorem} 
	\label{t:global}
	Let $u_0=(u_{10},...,u_{N0})\in H^2(\mathbb{T}^n )^N$ and make the same assumptions as in Lemma \ref{l:positivity}. Then in the solution from Theorem \ref{t:local}, we have $T_*=\infty$.  In other words,
	$u\in C^1((0,\infty), L^2(\mathbb{T}^n ))^N \cap C^0([0,\infty), H^2(\mathbb{T}^n ))^N. $
\end{theorem}
\begin{proof}
	Since $u_i(x,t)\geq 0$ for all $x,t$, we have, for each $t\geq 0$, 
	\begin{align}
		\|u_i\|_{L^1}=\int_{\mathbb T} u_i(x,t)dx.
	\end{align}
	However, the right-hand integral (total population) remains constant over time.  Therefore $\|u_i\|_{L^1}$ is constant over time.  Now recall the definition of  $T_*$, which states that if $\|u_i\|_{L^1}$ is bounded for all $i$ then $T_*=\infty$.
\end{proof}
This establishes global existence of weak positive solutions to Equation (\ref{eq:model2}).  To establish the analogous result for Equation (\ref{eq:model}), we note that any positive solution to Equation (\ref{eq:model2}) is also a positive solution to Equation (\ref{eq:model}), since $h(u_i(x,t))=u_i(x,t)$ whenever $u_i(x,t)\geq 0$.  Hence we have established the following. 
\begin{theorem} 
	\label{t:global2}
	The solution $u\in C^1((0,\infty), L^2(\mathbb{T}^n ))^N \cap C^0([0,\infty), H^2(\mathbb{T}^n ))^N$ from Theorem \ref{t:global}, is a positive, global, weak solution to Equation (\ref{eq:model}).  
\end{theorem}
In particular, in one spatial dimension the solutions are classical and strictly positive, as proved in the following.
\begin{theorem}\label{th:positivity1D}
	On $1$D domains, the  solution to Equation \eqref{eq:model} given in Theorem \ref{t:global2} is a classical strictly positive solution.
\end{theorem}
\begin{proof}
	In one spatial dimension we have the Sobolev embedding from $H^2$ to $C^1$. By using the same argument of \cite[Lemma 3.8]{giuntaetal2022a}, we can show that the solution $u$ given in Theorem \ref{t:global2} is such that $u(\cdot, t) \in C^2$. Therefore, in $1$D the solution to Equation \eqref{eq:model} satisfies
	\begin{equation}\label{eq:1dsol}
		u\in C^1((0,\infty), L^2(\mathbb{T}^n ))^N \cap C^0([0,\infty), C^2(\mathbb{T}^n ))^N,
	\end{equation}
	which is therefore a classical solution.
	To prove that this solution is strictly positive in 1D, we consider the following linear parabolic PDE problem
	
	\begin{equation}\label{eq:lin}
		\sigma_{it} = D_i \partial_{xx} \sigma_i +\partial_{x} \left( \sigma_i \sum_{j=1}^N \gamma_{ij} \partial_{x} (K_{ij} \ast u_j) \right), \qquad i=1,\dots,N,
	\end{equation}
	where $u =(u_1, \dots, u_N)$ is the solution to the one dimensional version of Equation \eqref{eq:model} satisfying \eqref{eq:1dsol}.
	Notice that the coefficients of the linear problem in Equation \eqref{eq:lin} are continuous. Let $\sigma =(\sigma_1, \dots, \sigma_N)$ a non-negative (component-wise) classical solution to Equation \eqref{eq:lin}. Then Harnack's inequality for parabolic systems (see \cite[Theorem 10, page 370]{evans2022partial})
	ensures that for each $0<t_1<t_2$ there exists a positive constant $C$ such that
	\begin{equation}\label{eq:hin}
		\sup_{\mathbb{T}^n} \sigma_i(x,t_1) \leq C\inf_{\mathbb{T}^n} \sigma_i (x,t_2), \qquad i=1,\dots,N.
	\end{equation}
	In particular, $u =(u_1, \dots, u_N)$ is a solution to Equation \eqref{eq:hin}, and therefore it satisfies the inequalities in Equation \eqref{eq:hin}, that is
	\begin{equation}\label{eq:hin_u}
		\sup_{\mathbb{T}^n} u_i(x,t_1) \leq C\inf_{\mathbb{T}^n} u_i (x,t_2),  \qquad i=1,\dots,N,
	\end{equation}
	for each $0<t_1<t_2$. Since $u_i \geq 0 $ and  $\lVert u_i\rVert=\lVert u_{i0}\rVert>0$,  it follows that $\sup u_i(x,t_1)>0$,  which implies that $\inf u_i (x,t)>0$ at any positive time $t$. 
	The above Harnack inequality is not available for weak solutions in higher dimensions, hence we prove strict positivity only for the 1D case.
\end{proof}


\section{Blow-up of the solutions in the local limit}\label{sec:blowup}
\noindent In this section we show that solution of the local version of Equation (\ref{eq:model}) (i.e., the equation obtained by choosing the kernels $K_{ij}$ equal to the $\delta$-Dirac function) can have finite time blow-up solutions for $n\geq 2$, where $n$ denotes the spatial dimension. To find such solutions we use an argument previously used for chemotaxis models (see \cite{Perthamebook}). Namely, we consider a case where an aggregation arises at a certain location and we orient the torus $\mathbb{T}^n$ in such a way that the `boundary' locations -- i.e. where $x_k=\pm L_k/2$ for some $k$ where $x=(x_1,\dots,x_n)$ -- are far away from this aggregation. Then we consider the second moment of this aggregate and show that for a bounded solution, the second moment becomes negative over time. This contradicts the assumption of the solution to be bounded and hence implies blow-up. We will consider two cases: $\gamma_{ij}<0$, for all $i,j=1, \dots, N$ (mutual attraction and self-attraction); $\gamma_{ii}<0$, $i=1, \dots, N$ (self-attraction) and $\gamma_{ij}>0$, $i \neq j$ (mutual avoidance).

\textbf{Case 1 .} Let $\gamma_{ij}<0$ (mutual attraction), for all $i,j=1, \dots, N$, and consider the PDE
\begin{equation}\label{eq:localmodel}
	u_{it} = D_i \Delta u_i +\nabla\cdot\left( u_i \sum_{j=1}^N \gamma_{ij} \nabla  u_j \right),
\end{equation}
obtained from Equation \eqref{eq:model} with $K_{ij}=\delta$, for all $i,j=1,\dots,N$, where $\delta$ denotes the $\delta$-Dirac distribution. 
Assume that ${u}(x,t)$ is the solution with initial condition ${u}_0=(u_{10}, \dots, u_{N0})$.  Assume further that, for all $i =1,\dots, N$, $u_{i0}$ decays to zero as $x_k \rightarrow -L_k,L_k$  for any $k=1, \dots, n$. We show that if
\begin{equation}
P:=\sum_{i=1}^N \int_{\mathbb{T}^n}  u_{i0}(x) dx > \frac{\lvert \mathbb{T}^n \rvert}{2}
\end{equation}
and if  
\begin{equation}
\gamma <- \frac{2 P D}{2P-\lvert \mathbb{T}^n \rvert},\quad  \text{ where }\quad \gamma:= \max_{i,j=1, \dots, N} \{\gamma_{ij}\} \quad  D:=\max_{i=1, \dots,N}\{D_i\},
\end{equation}
then the solution $u$ becomes unbounded in finite time.

Indeed, let $u=(u_{1}, \dots, u_{N})$ be a non-negative solution to system \eqref{eq:localmodel} which decays to zero as $x_k \rightarrow -L_k,L_k$ for any $i=1, \dots, N$ and $k=1,\dots,n$. Define the second moment as
\begin{equation}\label{eq:M}
M(t):= \sum_{i=1}^{N}\int_{\mathbb{T}^n} \lvert x \rvert^2 u_i(x,t) dx,
\end{equation}
and compute
\begin{equation}\label{eq:Mt}
\begin{aligned}
	\frac{d}{dt} M & =  \sum_{i=1}^{N}\int_{\mathbb{T}^n} \lvert x \rvert^2 u_{it} dx \\ &=  \sum_{i=1}^{N} D_i\int_{\mathbb{T}^n} \lvert x \rvert^2 \Delta u_{i}dx +  \sum_{i,j=1}^{N} \gamma_{ij} \int_{\mathbb{T}^n} \lvert x \rvert^2 \nabla \cdot (u_i  \nabla u_{j}) dx \\
	& = \sum_{i=1}^{N} D_i \left(\int_{\mathbb{T}^n}  \nabla \cdot ( \lvert x \rvert^2 \nabla u_i) dx -  \int_{\mathbb{T}^n}  \nabla ( \lvert x \rvert^2) \cdot \nabla u_{i} dx\right)dx + \sum_{i,j=1}^{N}  \gamma_{ij}\left(\int_{\mathbb{T}^n} \nabla \cdot ( \lvert x \rvert^2 u_i  \nabla u_{j})_x dx -  \int_{\mathbb{T}^n} \nabla  (\lvert x \rvert^2) \cdot ( u_i  \nabla u_{j}) dx \right)
	\\
	& = -\sum_{i=1}^{N} D_i   \int_{\mathbb{T}^n}  \nabla ( \lvert x \rvert^2) \cdot \nabla u_{i}dx - \sum_{i,j=1}^{N} \gamma_{ij}  \int_{\mathbb{T}^n}  \nabla  (\lvert x \rvert^2) \cdot ( u_i  \nabla u_{j}) dx \\
	& = -2\sum_{i=1}^{N} D_i   \int_{\mathbb{T}^n}  x  \cdot \nabla u_{i}dx - 2\sum_{i,j=1}^{N}\gamma_{ij}  \int_{\mathbb{T}^n}   x \cdot ( u_i  \nabla u_{j}) dx     \\
	& = -2\sum_{i=1}^{N} D_i   \int_{\mathbb{T}^n}  x  \cdot \nabla u_{i}dx - \sum_{i,j=1}^{N} \gamma_{ij} \int_{\mathbb{T}^n}   x \cdot \nabla( u_i   u_{j}) dx     \\
	& = -2\sum_{i=1}^{N} D_i   \int_{\mathbb{T}^n}  \nabla\cdot (x   u_{i})dx +2n\sum_{i=1}^{N} D_i   \int_{\mathbb{T}^n} u_i \,dx  - \sum_{i,j=1}^{N}\gamma_{ij}  \int_{\mathbb{T}^n}   \nabla\cdot ( x  u_i   u_{j}) \,dx + n\sum_{i,j=1}^{N} \gamma_{ij} \int_{\mathbb{T}^n}  u_i   u_{j} \,dx    \\
	& = 2n\sum_{i=1}^{N} D_i   \int_{\mathbb{T}^n} u_i \,dx   +n\sum_{i,j=1}^{N}\gamma_{ij}  \int_{\mathbb{T}^n}  u_i   u_{j} \,dx\\
	& \leq 2n\sum_{i=1}^{N} D_i   \int_{\mathbb{T}^n} u_i \,dx   + \gamma n \sum_{i,j=1}^{N}  \int_{\mathbb{T}^n}  u_i   u_{j} \,dx\\
	& = 2 n \sum_{i=1}^{N} D_i   \int_{\mathbb{T}^n} u_{i} dx + \gamma n  \int_{\mathbb{T}^n} \left(\sum_{i=1}^{N} u_i\right)^2 dx \\
	& \leq 2 D n \sum_{i=1}^{N}  \int_{\mathbb{T}^n} u_{i} dx + \gamma n  \int_{\mathbb{T}^n} \left(\sum_{i=1}^{N} u_i\right)^2   dx\\
	&  \leq 2 D n  \sum_{i=1}^{N}  \int_{\mathbb{T}^n} u_{i} dx + \gamma n \left(2\sum_{i=1}^{N} \int_{\mathbb{T}^n}  u_{i} dx - \lvert \mathbb{T}^n \rvert\right) \\
	&= 2 D n \sum_{i=1}^{N}  \int_{\mathbb{T}^n} u_{i0} dx + \gamma n \left(2\sum_{i=1}^{N} \int_{\mathbb{T}^n}  u_{i0} dx - \lvert \mathbb{T}^n \rvert \right)
\end{aligned} 
\end{equation}

where $\gamma:=\max_{i,j=1, \dots, N} \{\gamma_{ij}\}$ and $D:= \max_{i=1, \dots, N} \{D_i\}$. The third  and seventh equalities are obtained integrating by parts. The third inequality follows from $\gamma<0$ and the Young's inequality $a^2 \geq 2 a-1$.
The last equality follows from conservation of total mass, i.e. $\int_{\mathbb{T}^n}  u_{i}(x,t) dx = \int_{\mathbb{T}^n}  u_{i0}(x) dx, \text{ for all } t\geq 0$, where $u_{i0}$ is the initial condition.

By defining $ P = \sum_{i=1}^{N}  \int_{\mathbb{T}^n}  u_{i0}(x) dx $, Equation \eqref{eq:Mt} can be rewritten as
\begin{equation}
\begin{aligned}
	\frac{d}{dt} M &  \leq 2 D n P + \gamma n \left(2 P - \lvert \mathbb{T}^n \rvert\right).
\end{aligned}
\end{equation}
If 
\begin{equation}
P > \frac{\lvert \mathbb{T}^n \rvert}{2} \text{  and } \gamma = - \frac{2 P D+\varepsilon}{2P-\lvert \mathbb{T}^n \rvert},
\end{equation}
for some $\varepsilon>0$, then $\frac{d}{dt} M<-\varepsilon n$. Since the gradient of $M(t)$ is bounded above by a strictly negative constant, 
$-\varepsilon n$, there exists a finite time $T>0$ such that $M(T)=0$ (by Mean Value Theorem). Hence, by Equation \eqref{eq:M} and since $u_i(x,t)$ conserves its total mass, it follows that, for all $i=1, \dots, N$, $u_i$ tends to a finite weighted sum of Dirac delta functions as $t \rightarrow T$.

\textbf{Case 2.} Let $\gamma_{ii}<0$, for $i=1, \dots, N$, and $\gamma_{ij}\geq0$ (mutual avoidance), for $i \neq j$ and $i,j=1, \dots, N$, such that
\begin{equation}
\sum_{\substack{j=1\\j \neq i}}^{N} (\gamma_{ij}+\gamma_{ji})< -\gamma_{ii}, \text{ for all } i=1, \dots, N.
\end{equation}
We show that if
\begin{equation}
P:=\sum_{i=1}^N \int_{\mathbb{T}^n}  u_{i0}(x) dx > N\frac{\lvert \mathbb{T}^n \rvert}{2}
\end{equation}
and if  
\begin{equation}
\gamma <- \frac{4 P D}{2P-N\lvert \mathbb{T}^n \rvert},\quad  \text{ where }\quad \gamma:= \max_{i=1, \dots, N} \{\gamma_{ii}\} \quad  D:=\max_{i=1, \dots,N}\{D_i\},
\end{equation}
then the solution $u$ becomes unbounded in finite time. To this end, as in the previous case, we compute the time-derivative of the second moment of $M(t)$ as follows

\begin{equation}\label{eq:Mt2}
\begin{aligned}
	\frac{d}{dt} M & = 2n\sum_{i=1}^{N} D_i   \int_{\mathbb{T}^n} u_i \,dx   +n\sum_{i,j=1}^{N}\gamma_{ij}  \int_{\mathbb{T}^n}  u_i   u_{j} \,dx \\
	& = 2n\sum_{i=1}^{N} D_i   \int_{\mathbb{T}^n} u_i \,dx   +n\sum_{i=1}^{N}\gamma_{ii}  \int_{\mathbb{T}^n}  u_i^2  \,dx + n\sum_{\substack{i,j=1\\i \neq j}}^n \gamma_{ij} \int_{\mathbb{T}^n}  u_i   u_{j} \,dx
	\\ & \leq 2n\sum_{i=1}^{N} D_i   \int_{\mathbb{T}^n} u_i \,dx   +n\sum_{i=1}^{N}\gamma_{ii}  \int_{\mathbb{T}^n}  u_i^2  \,dx + \frac{n}{2}\sum_{\substack{i,j=1\\i \neq j}}^n (\gamma_{ij}+\gamma_{ji}) \int_{\mathbb{T}^n} u_i^2  \,dx \\
	& \leq 2n\sum_{i=1}^{N} D_i   \int_{\mathbb{T}^n} u_i \,dx   + \frac{n}{2}\sum_{i=1}^{N}\gamma_{ii}  \int_{\mathbb{T}^n}  u_i^2  \,dx   \\
	& \leq 2 D n\sum_{i=1}^{N}    \int_{\mathbb{T}^n} u_i \,dx   + \gamma \frac{n}{2} \sum_{i=1}^{N}  \int_{\mathbb{T}^n}  u_i^2 \,dx\\
	& \leq 2 D n \sum_{i=1}^{N}  \int_{\mathbb{T}^n} u_{i} dx + \gamma \frac{n}{2}\sum_{i=1}^{N}  \int_{\mathbb{T}^n}  (2 u_i-1) dx \\
	& = 2 D n \sum_{i=1}^{N}  \int_{\mathbb{T}^n} u_{i} dx + \gamma n \sum_{i=1}^{N}\int_{\mathbb{T}^n}  u_i dx - \gamma \frac{n}{2} N   \lvert \mathbb{T}^n \rvert\\
	&  = 2 D n \sum_{i=1}^{N}  \int_{\mathbb{T}^n} u_{i0} dx + \gamma n \sum_{i=1}^{N}\int_{\mathbb{T}^n}  u_{i0} dx  - \gamma \frac{n}{2} N   \lvert \mathbb{T}^n \rvert,
\end{aligned} 
\end{equation}
where $\gamma:=\max_{i=1, \dots, N} \{\gamma_{ii}\}$ and $D:= \max_{i=1, \dots, N} \{D_i\}$.
By defining $ P = \sum_{i=1}^{N}  \int_{\mathbb{T}^n}  u_{i0}(x) dx $, Equation \eqref{eq:Mt2} can be rewritten as
\begin{equation}
\begin{aligned}
	\frac{d}{dt} M &  \leq 2 D n P + \gamma n \left(P - \frac{N}{2}\lvert \mathbb{T}^n \rvert\right).
\end{aligned}
\end{equation}
If 
\begin{equation}
P > N\frac{\lvert \mathbb{T}^n \rvert}{2} \text{  and } \gamma = - \frac{4 P D+\varepsilon}{2P-N\lvert \mathbb{T}^n \rvert},
\end{equation}
for some $\varepsilon>0$, then $\frac{d}{dt} M<-\varepsilon n/2$. Then, by the same argument as {\bf Case 1}, $u_i$ tends to a finite weighted sum of Dirac delta functions as $t \rightarrow T$, for all $i=1, \dots, N$.

The above calculations do not give a complete categorisation of blow-up regimes, but do demonstrate the singular nature of non-linear and local cross-diffusion terms. As with chemotaxis models \cite{Horstmann,schmeiser,UsersGuide,bellomo2015}, such a categorisation requires advanced machinery such as energy estimates or multiscale arguments, which are beyond the scope of this work.  We leave the general question of blow-up in  Equation (\ref{eq:localmodel}) as an interesting open problem.

\section{Numerical simulations}\label{eq:num}

\noindent 
The aim of this section is to demonstrate numerically how the existence of solutions breaks down in the spatially-local limit of Equation \eqref{eq:model} for a few choice examples. For our numerical solutions, we use a spectral method and the numerical scheme described in \cite{giuntaetal2022a}. We also use the following averaging kernel
\begin{equation}
K_{ij}(x,y)=\begin{cases}
	\frac{\pi}{r_{ij}^2(\pi^2-4)}\left( 1+\cos\left( \frac{\pi}{r_{ij}}\sqrt{x^2+y^2} \right) \right), & \text{ if } x^2+y^2 \leq r_{ij}^2, \\
	0, & \text{ otherwise,}
\end{cases}
\end{equation}
which satisfies the assumptions of Theorem \ref{t:local}. 

Figure \ref{fig:attract} shows three sets of numerical simulations obtained by fixing $N=1$ population (in (a) and (b)), $N=2$ populations (in (c) and (d)) and $N=3$ populations (in (e) and (f)), with $\gamma_{ij}<0$ (mutual attraction) and $r_{ij}=r_{ji}$, for all $i,j$. The simulated populations become steeper as the detection radius, $r_{ij}$, decreases.  This is suggestive of blow-up as $r_{ij}$ vanishes, even though solutions remain bounded for all strictly positive $r_{ij}$. Note that in the limit $r_{ij}\rightarrow 0$ for all $i$ and $j$, this example reduces to \textbf{Case 1} analysed in the previous section.  There, we show that the local system may undergo a finite time blow-up for large enough initial data. 
The appearance of spike solutions becomes more pronounced with the addition of populations. In fact, for fixed values of $r_{ij}$, the addition of populations makes the solution higher and steeper, due to the fact that, in addition to self-attraction, the populations also exhibit mutual attraction (compare panels a,c, and e).

Figure \ref{fig:avoid} shows a similar analysis to Figure \ref{fig:attract}, but this time focusing on the situation relevant to \textbf{Case 2}, i.e. populations exhibiting mutual avoidance and a sufficiently strong self-attraction. As with Figure \ref{fig:attract}, we observe that as $r_{ij}$ decreases, the population profiles become steeper, suggesting blow-up as $r_{ij}$ tends to zero.
\begin{figure}[h] 
\centering
\textbf{Case 1.} $\gamma_{ij}<0$: Self attraction and Mutual attraction\par\medskip
\subfloat[$r_{11}=0.3$]
{ \includegraphics[width=0.16\textwidth]{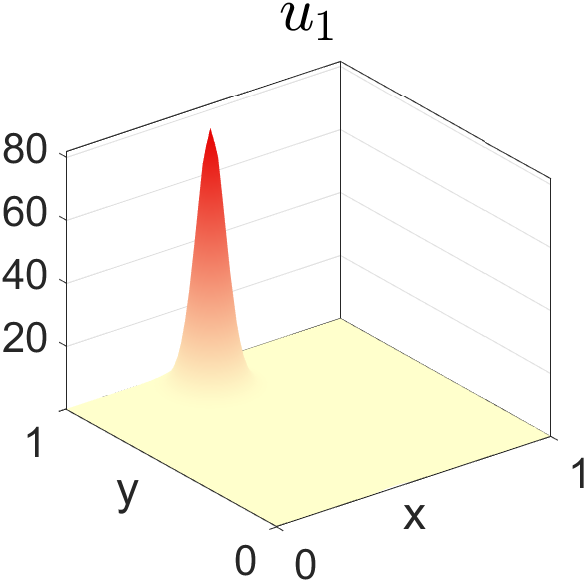}}
\hspace{0.25cm}
\subfloat[$r_{11}=0.2$]
{\includegraphics[width=0.16\textwidth]{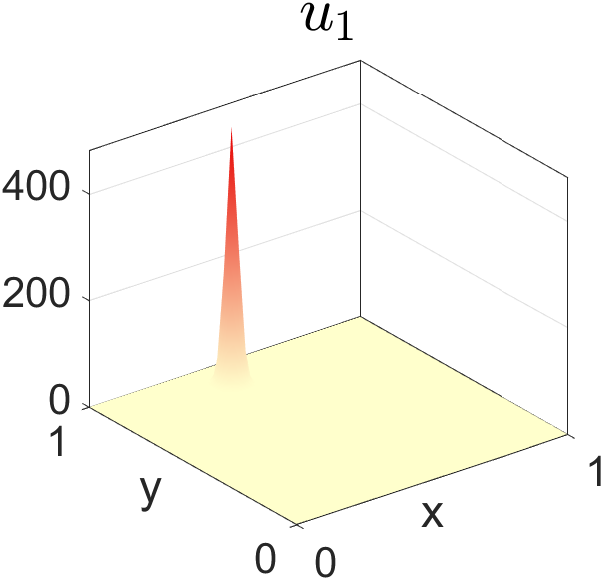}}\\
\subfloat[$r_{ij}=0.3$]
{ \includegraphics[width=0.16\textwidth]{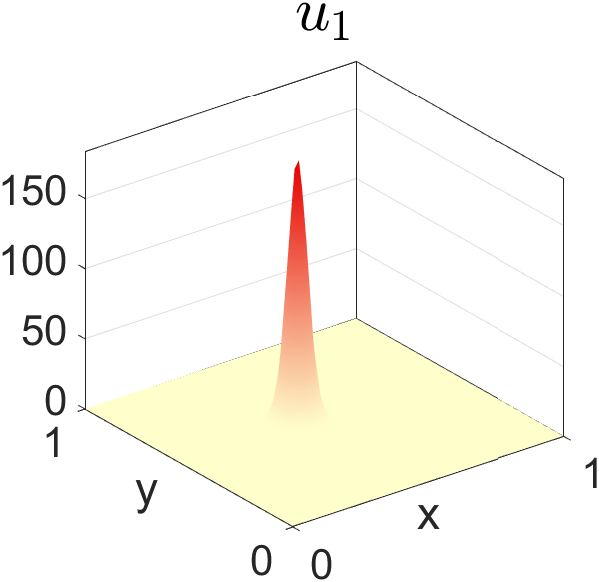}\hspace{-0.15cm}
	\includegraphics[width=0.16\textwidth]{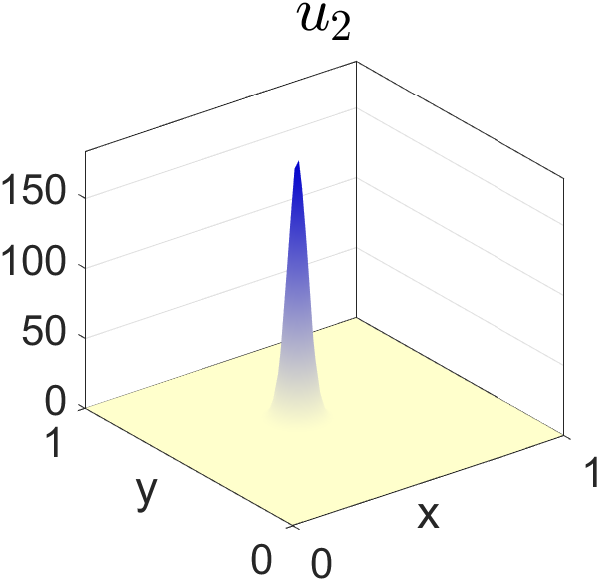}}\hspace{0.25cm}
\subfloat[$r_{ij}=0.2$]
{\includegraphics[width=0.16\textwidth]{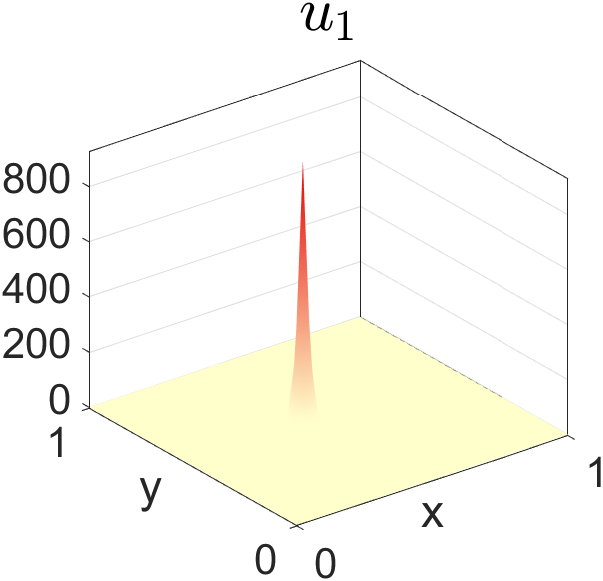}\hspace{-0.15cm}
	\includegraphics[width=0.16\textwidth]{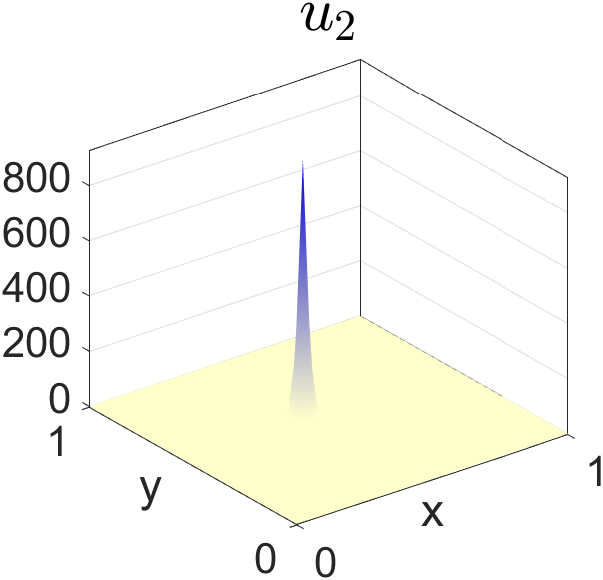}} \\
\subfloat[$r_{ij}=0.3$]
{\includegraphics[width=0.16\textwidth]{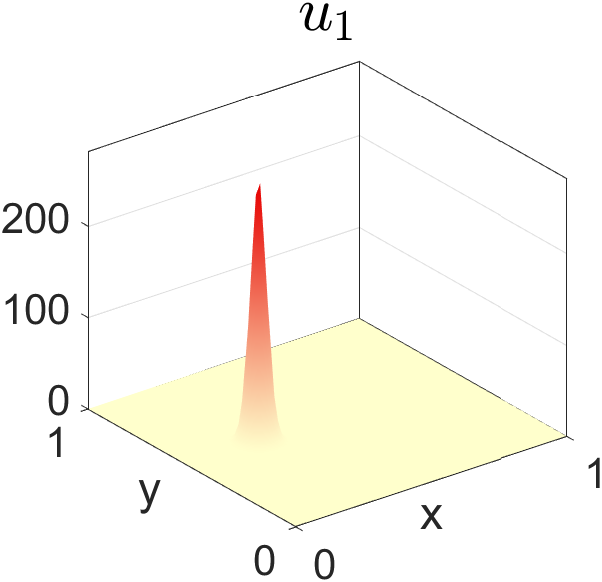}\hspace{-0.15cm}
	\includegraphics[width=0.16\textwidth]{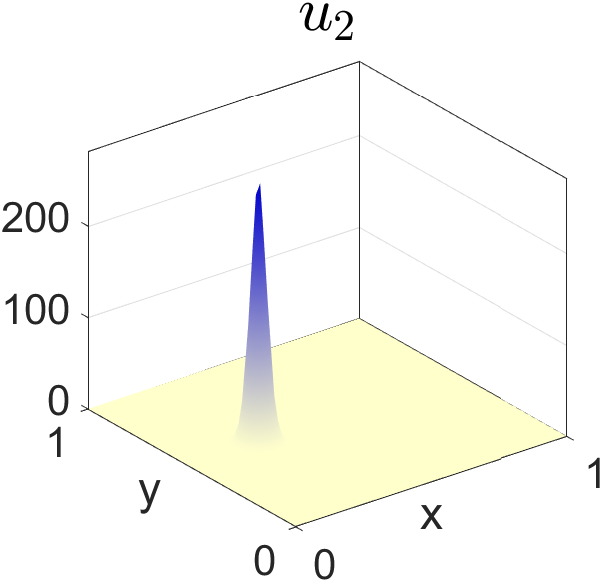}\hspace{-0.15cm}
	\includegraphics[width=0.16\textwidth]{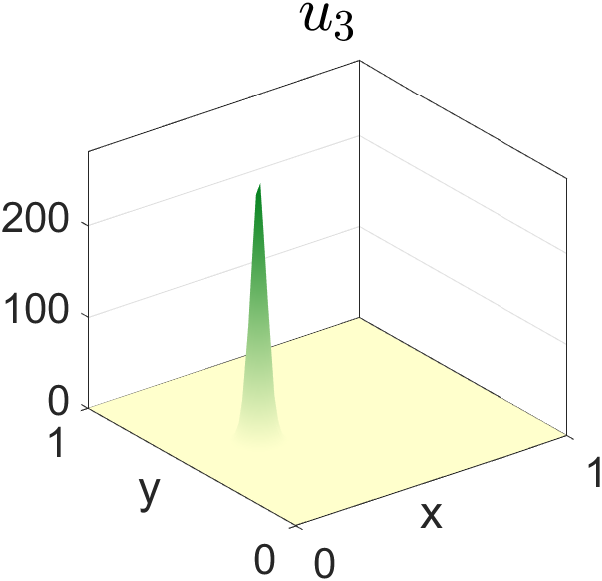}}\hspace{0.3cm}
\subfloat[$r_{ij}=0.2$]
{\includegraphics[width=0.16\textwidth]{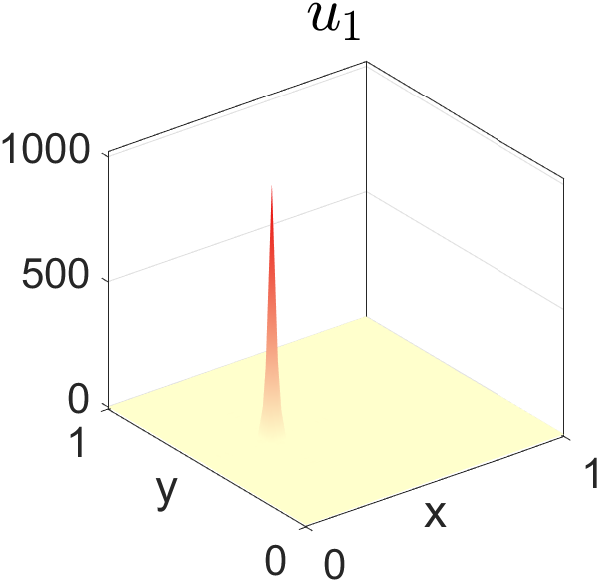}\hspace{-0.15cm}
	\includegraphics[width=0.16\textwidth]{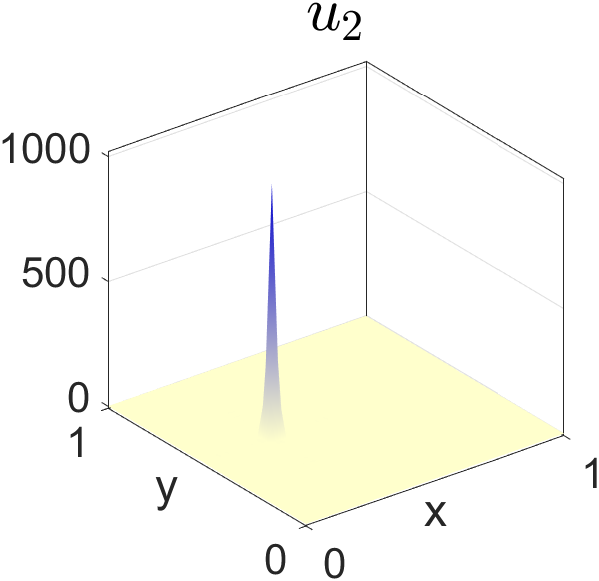}\hspace{-0.15cm}
	\includegraphics[width=0.16\textwidth]{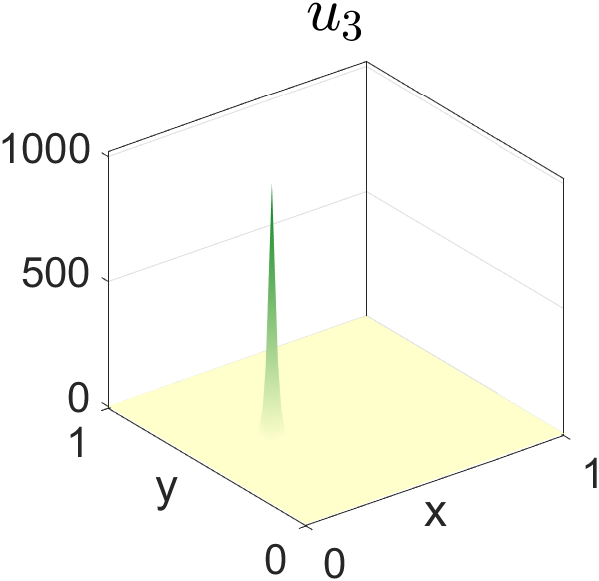}}
\caption{Numerical simulations of Equations \eqref{eq:model} on square domains with $N=1$ in (a)-(b), $N=2$ in (c)-(d), and $N=3$ in (e)-(f), for decreasing values of the sensing ranges $r_{ij}$, with $r_{ij}=r_{ji}$, for $i,j=1,2,3$. The other parameter values are: $D_i=1$, $\gamma_{ij}=\gamma_{ji}=-1 $, for all $i,j=1,2,3$.}
\label{fig:attract}
\end{figure}

\begin{figure}[h] 
\centering
\textbf{Case 2.} $\gamma_{ii}<0$, $\gamma_{ij}>0$ for $i\neq j$: Self attraction and Mutual avoidance\par\medskip
\subfloat[$r_{ij}=0.3$]
{\includegraphics[width=0.16\textwidth]{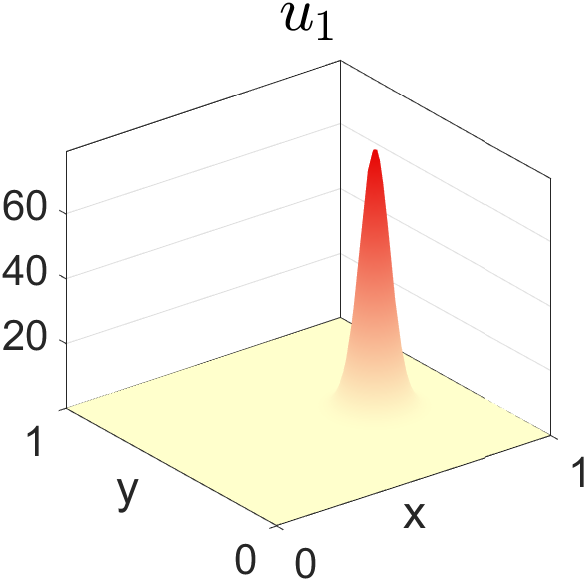}\hspace{-0.15cm}
	\includegraphics[width=0.16\textwidth]{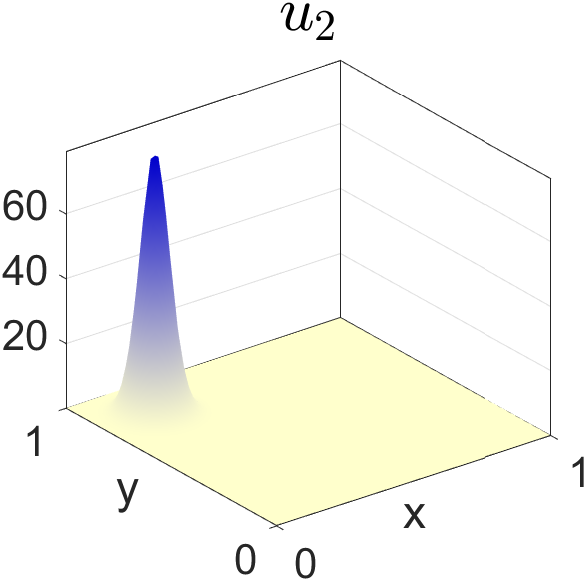}}\hspace{0.25cm}
\subfloat[$r_{ij}=0.2$]
{\includegraphics[width=0.16\textwidth]{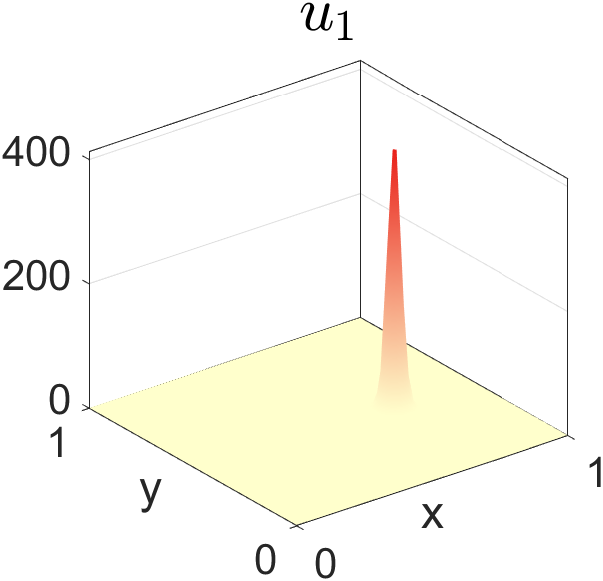}\hspace{-0.15cm}
	\includegraphics[width=0.16\textwidth]{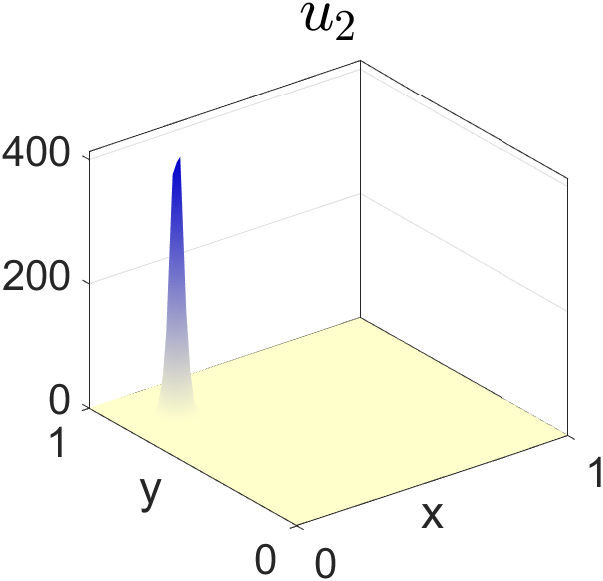}}\hspace{0.25cm}
\caption{Numerical simulations of Equations \eqref{eq:model} on square domains with $N=2$, for decreasing values of the sensing ranges $r_{ij}$, with $r_{ij}=r_{ji}$, for $i,j=1,2$. The other parameter values are: $D_1=D_2=1$, $\gamma_{11}=\gamma_{22}=-5$, $\gamma_{12}=\gamma_{21}=-1 $.}
\label{fig:avoid}
\end{figure}
In addition to these two examples inspired by the \textbf{Case 1} and \textbf{Case 2}, we also 
examine some cases where we do not currently have blow-up results.  For example, Figure \ref{fig:no-self} shows the case of two populations that attract each other ($\gamma_{12},\gamma_{21}<0$) but do not exhibit self-attraction ($\gamma_{ii}=0$). Likewise, in Figure \ref{fig:3} we consider a mixture of self-avoidance and self-attraction, mutual avoidance and mutual attraction, again observing a peak narrowing as $r_{ij}$ decreases. Note also that, in Figures \ref{fig:no-self} and \ref{fig:3}, only one of the detection radii is reduced, suggesting that solutions of System \eqref{eq:model} may blow-up even in situations where just one of the kernels $K_{ij}$ is the $\delta$-Dirac function. A detailed analysis of various blow-up scenarios is a fruitful direction of future research.


\begin{figure}[h] 
\centering
\textbf{Case 3.} $\gamma_{ii}=0$, $\gamma_{ij}<0$ for $i\neq j$: No-Self attraction and Mutual attraction\par\medskip
\subfloat[$r_{2j}=0.3$]
{\includegraphics[width=0.16\textwidth]{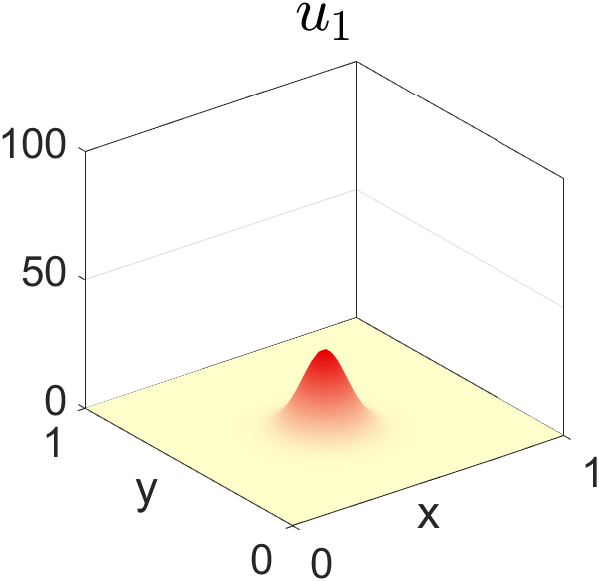}\hspace{-0.15cm}
	\includegraphics[width=0.16\textwidth]{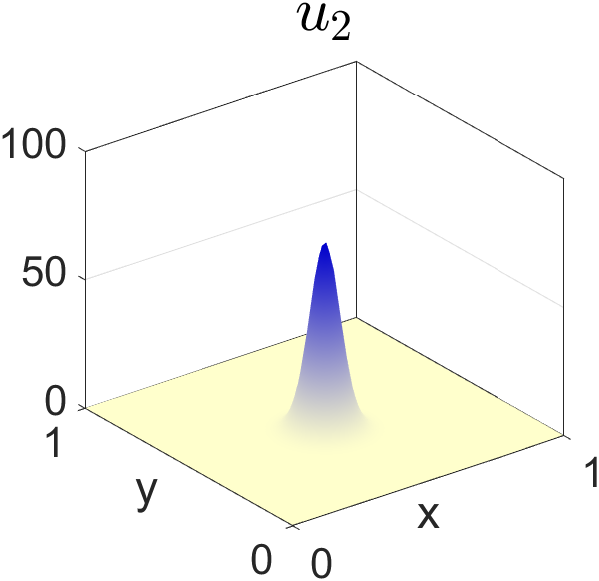}}\hspace{0.25cm}
\subfloat[$r_{2j}=0.1$]
{\includegraphics[width=0.16\textwidth]{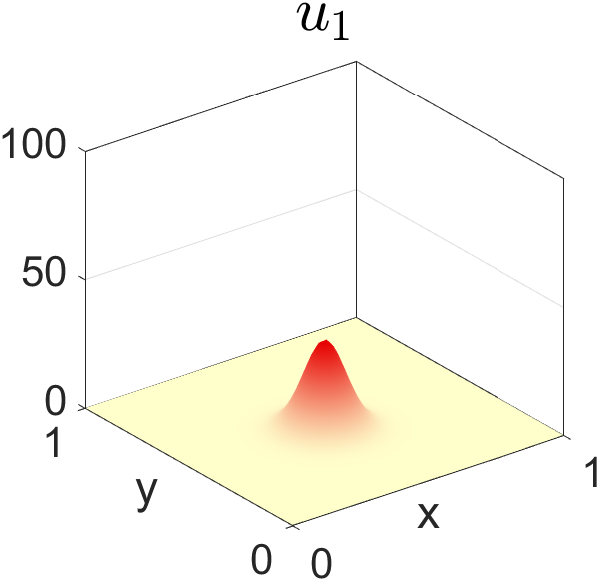}\hspace{-0.15cm}
	\includegraphics[width=0.16\textwidth]{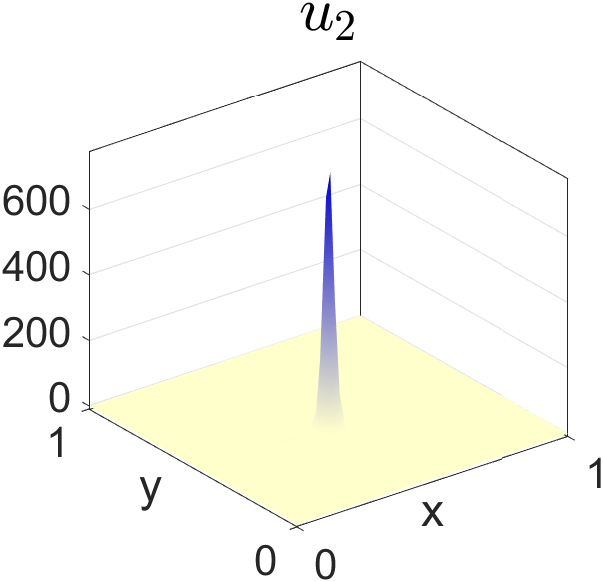}}\hspace{0.25cm}
\caption{Numerical simulations of Equations \eqref{eq:model} on square domains with $N=2$, for decreasing values of the sensing ranges $r_{21}=r_{22}$, with $r_{11}=r_{12}=0.4$ fixed. The other parameter values are: $D_1=D_2=1$, $\gamma_{11}=\gamma_{22}=0$, $\gamma_{12}=\gamma_{21}=-1.2 $.}
\label{fig:no-self}
\end{figure}

\begin{figure}[h] 
\centering
\textbf{Case 4.} Miscellaneous \par\medskip
\subfloat[$r_{3j}=0.3$]
{\includegraphics[width=0.16\textwidth]{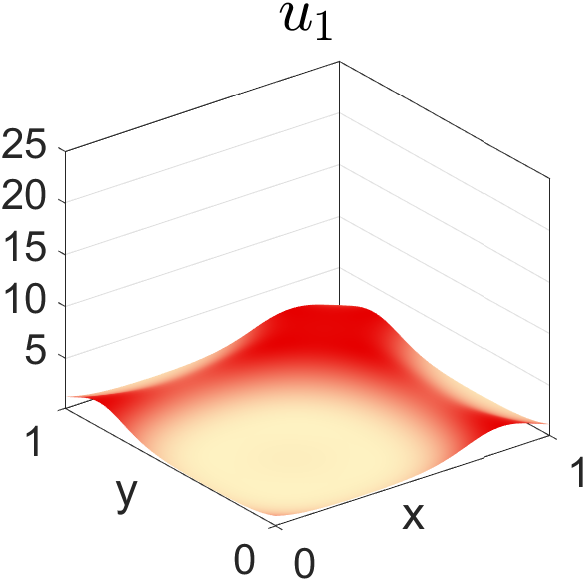}\hspace{-0.15cm}
	\includegraphics[width=0.16\textwidth]{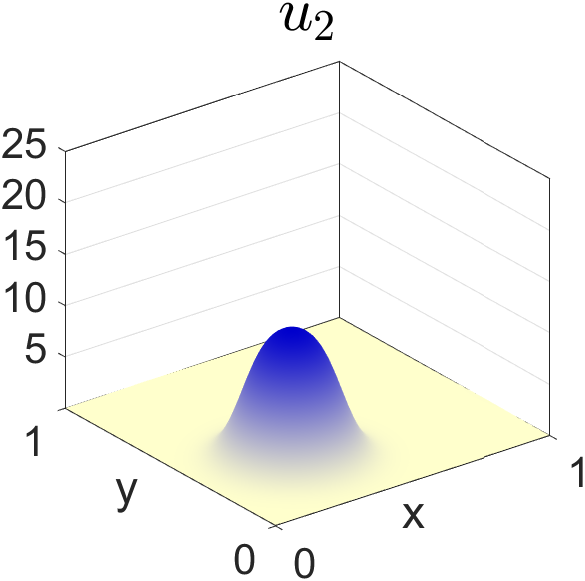}\hspace{-0.15cm}
	\includegraphics[width=0.16\textwidth]{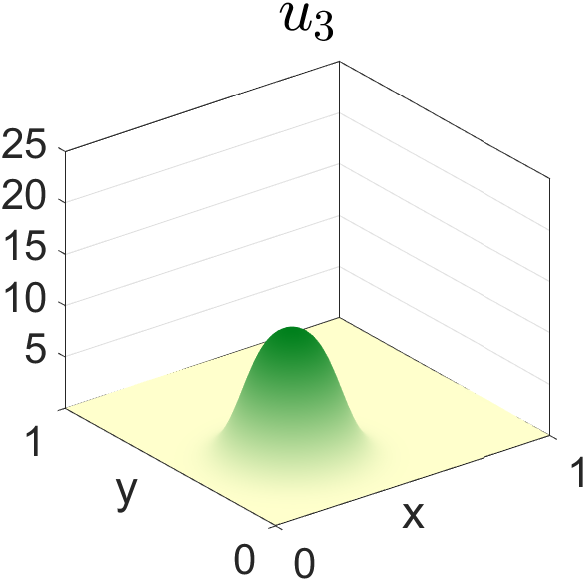}}\hspace{0.3cm}
\subfloat[$r_{3j}=0.1$]
{\includegraphics[width=0.16\textwidth]{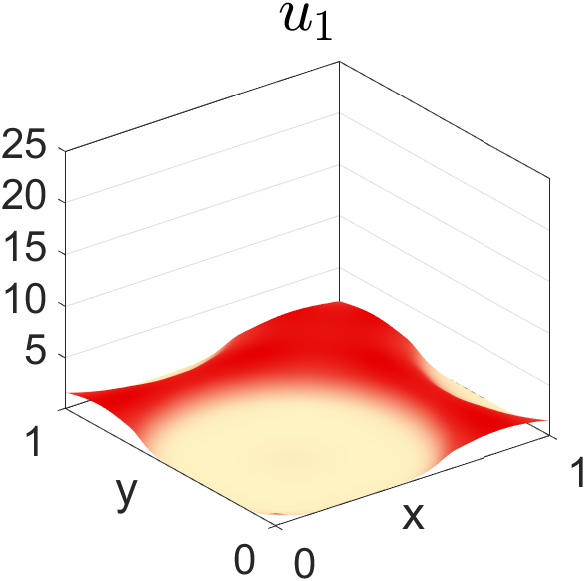}\hspace{-0.15cm}
	\includegraphics[width=0.16\textwidth]{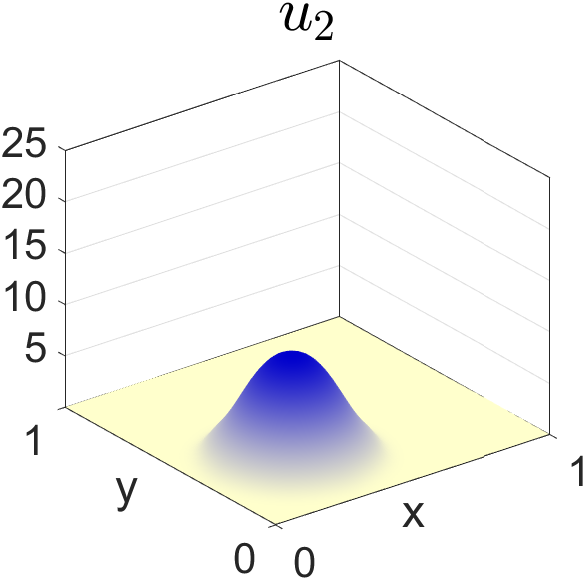}\hspace{-0.15cm}
	\includegraphics[width=0.16\textwidth]{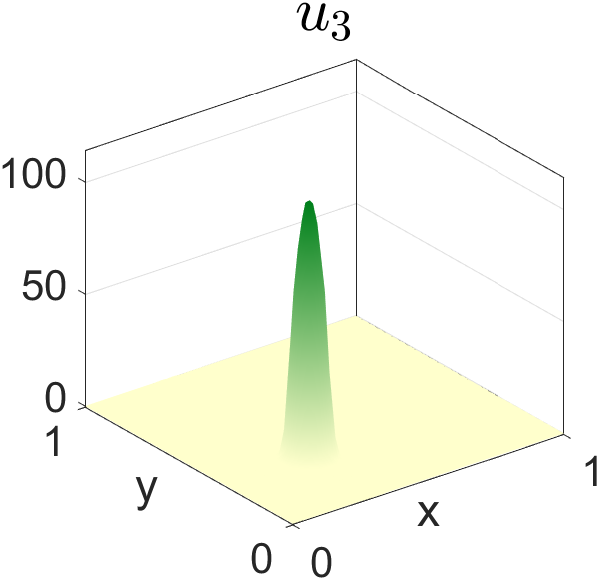}}\hspace{0.3cm}
\caption{Numerical simulations of Equations \eqref{eq:model} on square domains with $N=3$ for decreasing values of the sensing ranges $r_{31}=r_{32}=r_{33}$, while $r_{1j} =0.4$ and $r_{2j} =0.3$,for $j=1,2,3$, are kept fixed. The other parameter values are: $D_1=D_2=D_3=1$, $\gamma_{11}=1$, $\gamma_{22}=-1$, $\gamma_{33}=-1$, $\gamma_{12}=\gamma_{21}=1$, $\gamma_{13}=\gamma_{31}=1$, $\gamma_{23}=\gamma_{32}-1$.}
\label{fig:3}
\end{figure}


\section{Conclusions} 
We have established a comprehensive framework for understanding nonlocal advection-diffusion models of any number of interacting populations, which unifies and extends many previous results on existence of solutions, together with insights into blow-up of singular situations.  We have shown that, under the assumption of sufficiently smooth kernels,  positive solutions exist globally in any spatial dimension. This finding not only generalises existing knowledge, but also reveals a remarkable contrast with local models, where global existence often depends critically on the dimension of the spatial domain \cite{Gilbarg,Lieberman,Suzuki}. 

We also provide strong evidence for the critical role of nonlocal interactions in preventing the blow-up of solutions in finite time. 
Our analysis in Section \ref{sec:blowup} highlights the role of  local cross-diffusion terms and their ability to create sudden singularities in finite time. Similar phenomena have been observed and analysed in chemotaxis models \cite{horstmann2004,UsersGuide,bellomo2015}. While a complete categorisation of blow-up would require sophisticated tools beyond the scope of this paper, the analysis of local limits and numerical simulations provide solid support for the crucial role of nonlocality in preventing blow-up, paving the way for future explorations of the long-term behaviour and applications of these models in a variety of fields.

From a practical perspective, existence and blow-up results can be very useful in informing users of PDE models whether they are sensibly defined.  In particular, when performing numerics, knowledge of existence and blow-up regimes can inform whether those numerics are likely to produce meaningful results {\it a priori}, regardless of the numerical scheme being used. Here, we demonstrate how our insights on existence and blow-up translate to the appearance of spike-like solutions as the detection radius decreases to zero. As the limit is approached, it is necessary to use  ever-higher spatial resolution to capture the behaviour of the PDE accurately. However, away from this limit, solutions are nicely mollified, allowing for more rapid numerical analysis. 

The general global existence result presented here paves our way to a systematic analysis of nonlocal biological interactions. Our main interest is to gain a better understanding of animal space use and oriented animal movement. Possible applications of the model (\ref{eq:model}) are widespread, including  animal home ranges~\cite{briscoe2002home}, space use by territorial competitors \cite{potts2016territorial}, swarming and flocking \cite{eftimie2018hyperbolic}, species reactions to anthropogenic disturbances \cite{mokross2018can}, and biodiversity in heterogeneous environments \cite{tilman2014biodiversity}. The results presented here put us at ease to freely use nonlocal PDE models of type (\ref{eq:model}) to describe complex spatio-temporal interactions arising from such applications.

\vspace{3mm}

\noindent{\bf Acknowledgements:} JRP and VG acknowledge support of Engineering and Physical Sciences Research Council (EPSRC) grant EP/V002988/1 awarded to JRP. VG is also grateful for support from the National Group of Mathematical Physics (GNFM-INdAM).
TH is supported through a discovery grant of the Natural Science and Engineering Research Council of Canada (NSERC), RGPIN-2023-04269. MAL gratefully acknowledges support from NSERC Discovery Grant RGPIN-2018-05210 and from the Gilbert and Betty Kennedy Chair in Mathematical Biology.

\bibliographystyle{abbrv}
\bibliography{Globexist}
\end{document}